\documentclass[12pt, reqno]{amsart}
\usepackage{amsmath, amsthm, amscd, amsfonts, amssymb, graphicx, color}
\usepackage[bookmarksnumbered, colorlinks, plainpages]{hyperref}
\input{mathrsfs.sty}

\newtheorem{theorem}{Theorem}[section]
\newtheorem{lemma}[theorem]{Lemma}

\newtheorem{corollary}[theorem]{Corollary}
\theoremstyle{definition}

\theoremstyle{remark}

\begin{document}

\title[Estimates of operator convex and operator monotone functions]{Estimates of operator convex and operator monotone functions on bounded intervals}

\author[M. Fujii, M.S. Moslehian, H. Najafi, R. Nakamoto]{Masatoshi Fujii$^1$, Mohammad Sal Moslehian$^{2}$, Hamed Najafi$^{2}$ and Ritsuo Nakamoto$^{3}$}

\address{$^{1}$ Department of Mathematics, Osaka Kyoiku University, Kashiwara, Osaka 582-8582, Japan.}
\email{mfujii@cc.osaka-kyoiku.ac.jp}

\address{$^{2}$ Department of Pure Mathematics, Center of Excellence in Analysis on Algebraic Structures (CEAAS), Ferdowsi University of Mashhad, P.O.Box 1159, Mashhad 91775, Iran }
\email{ Moslehian@ferdowsi.um.ac.ir, moslehian@member.ams.org (Moslehian), hamednajafi@gmail.com  (Najafi)}

\address{$^{3}$ 3-4-13, Daihara-cho, Hitachi 316-0021, Japan.}
\email{r-naka@net1.jway.ne.jp}



\begin{abstract}
Recently the behavior of operator monotone functions on unbounded intervals with respect to the relation of strictly positivity has been investigated. In this paper we deeply study such behavior not only for operator monotone functions but also for operator convex functions on bounded intervals. More precisely, we prove that if $f$ is a nonlinear operator convex function on a bounded interval $(a,b)$ and $A, B$ are bounded linear operators acting on a Hilbert space with spectra in $(a,b)$ and $A-B$ is invertible, then $sf(A)+(1-s)f(B)>f(sA+(1-s)B)$. A short proof for a  similar known result concerning a nonconstant operator monotone function on $[0,\infty)$ is presented. Another purpose is to find a lower bound for $f(A)-f(B)$, where $f$ is a nonconstant operator monotone function, by using a key lemma. We also give an estimation of the Furuta inequality, which is an excellent extension of the L\"owner--Heinz inequality.
\end{abstract}

 \maketitle


\section{Introduction}

Let $(\mathscr{H}, \langle \cdot,\cdot\rangle)$ be a complex Hilbert
space and $\mathbb{B}(\mathscr{H})$ denote the algebra of all
bounded linear operators on $\mathscr{H}$ equipped with the operator
norm $\|\cdot\|$. An operator $ A\in \mathbb{B}(\mathscr{H})$ is called {\it positive} if $\langle Ax, x\rangle \geq 0$ holds for every $x\in \mathscr{H}$ and then we write $A\geq 0$. For self-adjoint operators $A,B \in \mathbb{B}(\mathscr{H})$, we say $A\leq B$ if $B-A\geq0$. Further, we write $A > B$ if $A\geq B$ and $A-B$ is invertible. When $A>0$, we call $A$ strictly positive.

Let $f$ be a real-valued function defined on an interval $J$. If for any self-adjoint operators $A, B \in \mathbb{B}(\mathscr{H})$ with
spectra in $J$,
\begin{itemize}
\item   $A\leq B$ implies $f(A)\leq f(B)$, then $f$ is said to be {\it operator monotone};\\
\item   $f(\lambda A + (1-\lambda)B) \leq \lambda f(A)+(1-\lambda)f(B)$ for all $\lambda \in [0,1]$, then $f$ is said to be {\it operator convex}.
\end{itemize}
Let us state our main terminology.\\

{\bf Definition.} A continuous real valued function $f$ defined on an interval $J=(a,b)$ is called
\begin{itemize}
\item [(i)] {\it strictly operator monotone} if  $A < B$ implies $f(A)< f(B)$ for all self-adjoint operators $A, B \in \mathbb{B}(\mathscr{H})$ with spectra in $J$.
\item [(ii)] {\it strictly operator convex} if $f(\lambda A+ (1-\lambda)B)< \lambda f(A)+ (1-\lambda) f(B)$
 for all $0 < \lambda < 1$ and all selfadjoint operators $A, B $ with spectra in $J$ such that $A-B$ is invertible.
\end{itemize}
Kwong \cite{KWO} showed that $f(t)=t^r$ is strictly operator monotone whenever $0<r\leq 1$. In addition, Uchiyama\cite{U2000} studied strictly operator convex functions.

Recall that if $f$ is an operator monotone function on $[0,\infty)$, then $f$ can be represented as
$$f(t)=f(0)+\beta t + \int_0^\infty \frac{\lambda t}{\lambda+t} d\mu(\lambda)\,,$$
where $\beta \geq 0$ and $\mu$ is a positive measure on $[0,\infty)$ and if $f$ is an operator convex function on $[0,\infty)$, then $f$ can be represented as
$$f(t)=f(0)+\beta t +\gamma t^2 +\int_0^\infty \frac{\lambda t^2}{\lambda+t} \ d\mu(\lambda)\,,$$
in which $\gamma \geq 0$, $\beta= f_+^{'}(0)=\lim_{t\to
0^+}\frac{f(t)-f(0)}{t}$ and $\mu$ is a positive measure on
$[0,\infty)$. The integral representations of operator convex and operator monotone functions on bounded intervals are different as we will see later.

The L\"owner--Heinz inequality is one of the most important facts in the theory of operator inequalities. It says that the function $t^p$ is operator monotone for $p \in [0,1]$; cf. \cite{HEI}. Recently the behavior of operator monotone functions on unbounded intervals with respect to the relation of strictly positivity has been investigated. In \cite{MN}, an estimation of the L\"owner--Heinz inequality was proposed as follows.

{\bf Theorem A.} {\it
If $ A > B \geq 0$ and $0 < r \le 1$, then $ A^r - B^r \geq \|A\|^r - (\|A\| - m)^r > 0$, and $ \log A - \log B \geq \log \|A\| - \log (\|A\| - m) > 0$,
where $m = \|(A-B)^{-1}\|^{-1}$.}

\vspace{2mm}

Very recently, the following generalization of Theorem A is given in \cite{FKN}, see also \cite{BJMA}.

{\bf Theorem B.} \ {\it If $ A > B \geq 0$ and $f$ is a non-constant operator monotone function on $[0, \infty)$, then $ f(A) - f(B) \geq f(\|B\|+ m) - f(\|B\|) > 0$, where $m= \|(A-B)^{-1}\|^{-1}$.}

\vspace{2mm}

As a consequence, we have the following improvement of Theorem A.

\vspace{2mm}

{\bf Theorem C.} \ {\it
If $ A > B \geq 0$ and $0 < r \le 1$, then
$$ A^r - B^r \geq (\|B\|+ m)^r - \|B\|^r > 0 $$
and $ \log A - \log B \geq \log (\|B\|+ m) - \log \|B\| > 0$,
where $m= \|(A-B)^{-1}\|^{-1}$.
}

The first aim of this paper is to prove that a nonlinear operator convex function on a bounded interval is strictly operator convex. The second one is to give a precise consideration of Theorem B for operator monotone functions on a finite interval. The third purpose is to extend Theorem C to the Furuta inequality. We recall the unforgettable fact that the Furuta inequality is a beautiful extension of the L\"owner--Heinz inequality.


\section{Strictly operator convex functions}

In this section we treat the behavior of operator convex functions on bounded intervals with respect to the relation of strictly positivity. To this end we need some lemmas.

\begin{lemma}\label{1-1}  
 The function $f(x)=\frac 1x$ is strictly operator convex on $(0, \infty)$, that is, if $A, B$ are strictly positive operators such that $A-B$ is invertible, then for each $0 < s < 1$,
 $$ s A^{-1} + (1-s) B^{-1} > (s A + (1-s)B)^{-1}.$$
\end{lemma}
\begin{proof}
Let $s \in (0,1)$. Put $H= A^{\frac{1}{2}} B^{-1} A^{\frac{1}{2}}$ and
$K=s A^{-1} + (1-s) B^{-1} - (s A + (1-s)B)^{-1}$. Then
$$
A^{\frac{1}{2}}KA^{\frac{1}{2}} = s + (1-s)H - (s + (1-s)H^{-1})^{-1}=f(H),
$$
where  $f(x)= s + (1-s)x - (s + (1-s)x^{-1})^{-1}$ for $x>0$.
It is clear that $f(x)> 0$ for $x>0$ except $x=1$.

Now, since $A-B$ is invertible, so is $1-H$ and thus $1$ does not belong to the spectrum of $H$.  Hence we have $f(H) > 0$, so $K > 0$, which is the desired inequality.
\end{proof}

\begin{lemma}\label{1-2}   
For each $\lambda$ with $|\lambda| \leq 1$, the function $f_{\lambda} (x) = \frac{x^2}{1- \lambda x}$ is strictly operator convex on $J=(-1,1)$.

\end{lemma}

\begin{proof}
We show that if $A, B$ are selfadjoint operators with spectra in $J$ such that $A-B$ is invertible, then
$$s f_{\lambda} (A) + (1-s) f_{\lambda} (B) > f_{\lambda}(s A + (1-s)B)$$
holds for $0<s<1$.
Fix $\lambda \not= 0$.  Since $f_{\lambda} (x) = \frac{-x}{\lambda}+ \frac{-1}{\lambda^2} + \frac{1}{\lambda^2(1- \lambda x)}$, we have
\begin{align*}
&\lambda^2(s f_{\lambda} (A) + (1-s) f_{\lambda} (B) - f_{\lambda}(s A + (1-s)B)) \\
& = -s \lambda A -s + s(1- \lambda A)^{-1}
 -(1-s)\lambda B -(1-s) + (1-s)(1- \lambda B)^{-1} \\
&\quad  + \lambda (s A + (1-s)B)+ 1 - (1- \lambda (s A + (1-s)B))^{-1} \\
& = s(1- \lambda A)^{-1} + (1-s)(1- \lambda B)^{-1} - (1- \lambda (s A + (1-s)B))^{-1}.
\end{align*}
Applying Lemma \ref{1-1} for $ 1- \lambda A $ and $ 1- \lambda B $, we reach the conclusion.

Incidentally, for the case $\lambda = 0$, i.e., $f_0(x)=x^2$, we have
\begin{align*}
&s A^{2} + (1-s) B^{2} - (s A + (1-s)B)^{2} \\
&= s A^{2} + (1-s) B^{2} - s^2 A^{2} - (1-s)^2 B^{2} - s(1-s) (AB + AB)\\
&= s(1-s) (A-B)^2 >0.
\end{align*}
\end{proof}

\begin{theorem}\label{1-0}  
 Any nonlinear operator convex function on a finite interval $(a,b)$  is strictly operator convex on $(a,b)$.
\end{theorem}

\begin{proof}
We may assume that $(a,b)=(-1,1)$.
 It is known that any operator convex function $f$ on $(-1,1)$ can be
represented as
\begin{eqnarray}\label{7-3}
f(x)=f(0)+\alpha x + \int_{-1}^{1}\frac{x^2}{1-\lambda x} d\mu(\lambda),
\end{eqnarray}
where $\alpha \in \mathbb{R}$ and $\mu$ is a positive measure on $[-1,1]$, \cite[Theorem 4.5]{HansenPedersen1982}. Since $f$ is nonlinear, we infer that $\mu \neq 0$.
By Lemma \ref{1-2} we conclude that
 $$ s f(A)+ (1-s) f(B) > f(sA+ (1-s)B) $$
for each $0 < s < 1$ and selfadjoint operators $A, B$ with spectra in $(-1,1)$ such that $A-B$ is invertible.
\end{proof}

\section{Estimates of Operator Monotone Functions}

We start this section with the following lemma. The first part borrow from \cite{MN} and the second is another variant of it.

\begin{lemma}\label{ess}
Let $A>B>0$ and $m=\|(A-B)^{-1}\|^{-1}$. Then
\begin{itemize}
\item[(i)] $B^{-1} - A^{-1} \geq \frac 1{\|A\|-m} - \frac 1{\|A\|}$;
\item[(ii)] $B^{-1}-A^{-1}\geq {m\over{(\|B\|+m)\|B\|}}$.
\end{itemize}
\end{lemma}
\begin{proof} \ (i) See Lemma 2.1 of \cite{MN}.\\
(ii) Because of $A-B \geq m$ we have
$$ B^{-1}-A^{-1} \geq B^{-1}-(B+m)^{-1}=mB^{-1}(B+m)^{-1} \geq \frac m{\|B\|(\|B\|+m)}.
$$
\end{proof}

Now we note that  $m=\|(A-B)^{-1}\|^{-1}$ is the maximum among $c \geq 0$ such that $A-B \geq c$.

Let us give an alternative proof of Theorem B by using integral representation of operator monotone functions on $[0,\infty)$.

\begin{proof}[Proof of Theorem B] \
Note that $f$ admits the the following integral representation:
$$ f(t)= a+bt+\int_{-\infty}^0 \frac {1+ts}{s-t} dm(s)
= a+bt+\int_{-\infty}^0 (-s-\frac {1+ s^2}{t-s}) dm(s)
 $$
where $b \geq 0$ and $m(s)$ is a positive measure. Hence it follows from Lemma 3.1 
 (ii) that
\begin{align*}
f(A)-f(B)&= b(A-B)+\int_{-\infty}^0 (1+s^2)((B-s)^{-1}-(A-s)^{-1}) dm(s) \\
   &\geq bm+\int_{-\infty}^0 (1+s^2)\left(\frac 1{\|B\|-s}-\frac 1{\|B\|-s+m}\right) dm(s) \\
   &= f(\|B\|+m)-f(\|B\|) \ (>0).
   \end{align*}
\end{proof}

Finally we propose an explicit lower bound of the difference $f(A) - f(B)$ for operator monotone functions on a finite interval $(a,b)$ when $A>B$; see \cite[Proposition 2.2]{MN}. Let us $m_X = \min {\rm sp}(X)$ and $M_X = \max {\rm sp}(X)$ for a self-adjoint operator $X$, where ${\rm sp}(X)$ denotes the spectrum of $X$. Note if $-1 < \lambda <0$ and $-1 \leq X \leq 1$, then $1-\lambda X\geq 0$. Hence
$$ \|1-\lambda X\| = \max \{1-\lambda t: t \in {\rm sp}(X)\}=1-\lambda \max \{t: t \in {\rm sp}(X)\}= 1-\lambda M_X.$$
Similarly if $0 < \lambda <1$ and $-1 \leq X \leq 1$, then $\|1-\lambda X\| = 1-\lambda m_X$.

To achieve our main result we need the following key lemma. For $\lambda \in (-1,1)$, set
$$f_{\lambda}(t) := \frac t{1-\lambda t},\qquad t\in (-1,1)\,.$$

\begin{lemma} \
If $-1 < B < A < 1$ and $A-B \ge m >0$, then
$$
f_\lambda(A) - f_\lambda(B) \ge
\begin{cases}
f_\lambda(M_B+ m) - f_\lambda(M_B) \quad \text{for} \ -1<\lambda \le 0 \\
f_\lambda(m_A) - f_\lambda(m_A-m) \ \quad \text{for} \ \ 0 < \lambda < 1.
\end{cases}
$$
\end{lemma}

\begin{proof} \
First of all, we have
$$ f_0(A) - f_0(B) = A - B \ge m = f_0(M_B+m) - f_0(M_B).  $$
Next suppose that $-1 < \lambda < 0$.  Since
$$ (1- \lambda A) -  (1- \lambda B) = -\lambda (A-B) \ge -\lambda m > 0, $$
it follows from Lemma \ref{ess} that
\begin{align*} (1- \lambda B)^{-1} -  (1- \lambda A)^{-1} & \ge
   \|1- \lambda B\|^{-1} - (\|1- \lambda B\|-\lambda m)^{-1}  \\
& = (1- \lambda M_B)^{-1} - (1- \lambda M_B-\lambda m)^{-1}\,,
\end{align*}
whence
\begin{align*}
f_\lambda(A) - f_\lambda(B) &= \frac 1{-\lambda}((1- \lambda B)^{-1} -  (1- \lambda A)^{-1}) \\
& \ge \frac 1{-\lambda}((1- \lambda M_B)^{-1} - (1- \lambda (M_B+ m))^{-1}) \\
&= f_\lambda(M_B+ m) - f_\lambda(M_B).
\end{align*}
Finally suppose that $0 < \lambda < 1$.  Since
$$ (1- \lambda B) -  (1- \lambda A) = \lambda (A-B) \ge \lambda m > 0, $$
it follows from Lemma \ref{ess} that
\begin{align*} (1- \lambda A)^{-1} -  (1- \lambda B)^{-1} & \ge
   \|1- \lambda A\|^{-1} - (\|1- \lambda A\|+\lambda m)^{-1}  \\
& = (1- \lambda m_A)^{-1} - (1- \lambda m_A+\lambda m)^{-1}.
\end{align*}
Hence we have
\begin{align*}
f_\lambda(A) - f_\lambda(B) &= \frac 1{\lambda}((1- \lambda A)^{-1} -  (1- \lambda B)^{-1}) \\
& \ge \frac 1{\lambda}((1- \lambda m_A)^{-1} - (1- \lambda (m_A- m))^{-1}) \\
&= f_\lambda(m_A) - f_\lambda(m_A-m).
\end{align*}
\end{proof}

\begin{theorem}
Any non-constant operator monotone function on a finite interval $J$ is strictly operator monotone on $J$.
\end{theorem}

\begin{proof}  We may assume that $J=(-1,1)$. Put $f_{\lambda}(t)= \frac t{1-\lambda t}$ for $\lambda \in (-1,1)$.
An operator monotone function on $(-1, 1)$ is represented as
$$f(t) = f(0) + f^{\prime}(0) \int_{-1}^1 f_{\lambda}(t) d\mu(\lambda), $$
where $\mu$ is a nonzero positive measure on $(-1, 1)$. Since $f$ is nonconstant, we have   $f^{\prime}(0) > 0$.
We here decompose $f$ as $f = f(0) + g_1 + g_2$, where
$$ g_1(t) = f^{\prime}(0) \int_{-1}^0 f_{\lambda}(t) d\mu(\lambda)\quad \hbox{and}\quad
 g_2(t) = f^{\prime}(0) \int_{0}^1 f_{\lambda}(t) d\mu(\lambda).
 $$
 Then $g_1$ (resp., $g_2$) is an operator concave (resp. operator convex) increasing function.
Hence it follows from Lemma 3.2 that
\begin{align*}
 f(A) - f(B) &= f^{\prime}(0) \int_{-1}^1 (f_{\lambda}(A) - f_{\lambda}(B))d\mu(\lambda) \\
      &=  g_1(A) - g_1(B) +  g_2(A) - g_2(B)       \\
      & \ge g_1(M_B+\epsilon) - g_1(M_B) +  g_2(m_A) - g_2(m_A-\epsilon)>0.
\end{align*}
\end{proof}


\section{Furuta inequality}

First of all, we cite {\it the Furuta inequality (FI)} established in \cite{FI} for reader's convenience, see also \cite{Fujii1990}, \cite{Furuta1989}, \cite{Kamei1988} and
\cite{Tanahashi1996} for the best possibility of it. It says that if $A \ge B \ge 0$, then for each $r \ge 0$,
$$A^{\frac {p+r}q} \ge (A^{\frac r2}B^pA^{\frac r2})^{\frac 1q} $$
holds for $p \ge 0, \ q \ge 1$ with
$$ (1+r)q \ge p+r. $$

To extend Theorem B, we remark that the case $r=0$ in (FI) is just the L\"owner--Heinz inequality.
Now we introduce a constant $k(b,m,p,q,r)$ for $b,m,p,q,r \ge 0$ by
$$k(b,m,p,q,r)= (b+m)^{\frac {p+r}q-r}-b^{\frac {p+r}q-r}.     $$
As a matter of fact, we have an extension of Theorem B in the form of Furuta inequality.
We denote by $m_A= \|A^{-1}\|^{-1}$, the minimum of the spectrum of $A$.

\begin{theorem} \ Let $A$ and $B$ be invertible positive operators with $A-B\geq m>0$. Then for $0<r\leq 1$,
$$A^{\frac {p+r}q}-(A^{\frac r2}B^pA^{\frac r2})^{\frac 1q} \geq k(\|B\|,m,p,q,r){m_A}^{r}$$
holds for $p \ge 0, \ q \ge 1$ with
$ (1+r)q \ge p+r \ge qr.$
\end{theorem}

\vskip.2cm
\begin{proof} \ We note that $q \ge 1$ and $ (1+r)q \ge p+r \ge qr$ assure that the exponent $\frac {p+r}q-r$ in the constant $k(b,m,p,q,r)$ belongs to $[0,1]$. Since $0 \le r \leq 1,$ it follows from Theorem B that
\begin{align*}
A^{\frac {p+r}q}-(A^{\frac r2}B^pA^{\frac r2})^{\frac 1q} &=A^{\frac {p+r}q}-A^{\frac r2}B^{\frac p2}(B^{\frac p2}A^r B^{\frac p2})^{\frac 1q -1}B^{\frac p2}A^{\frac r2} \\
     &=A^{\frac {p+r}q}-A^{\frac r2}B^{\frac p2}(B^{\frac {-p}2}A^{-r} B^{\frac {-p}2})^{1- \frac 1q}B^{\frac p2}A^{\frac r2} \\
     &\geq A^{\frac {p+r}q}-A^{\frac r2}B^{\frac p2}(B^{-\frac p2}B^{-r} B^{-\frac p2})^{1-\frac 1q}B^{\frac p2}A^{\frac r2} \\
     &=A^{\frac {p+r}q}-A^{\frac r2}B^{p-(p+r)(1-\frac 1q)}A^{\frac r2} \\
     &=A^{\frac r2}(A^{\frac {p+r}q-r}-B^{\frac {p+r}q-r})A^{\frac r2}  \\
     &\ge k(\|B\|,m,p,q,r)A^r \\
     &\ge k(\|B\|,m,p,q,r) {m_A}^{r}.
    \end{align*}
\end{proof}

Next we cosider the optimal case $q=\frac {1+r}{p+r}$ with $p \ge 1$, which is the most important in the Furuta inequality, by virtue of the L\"owner--Heinz inequality.
The proof of the following theorem is as same as that of the Furuta inequality.

\begin{theorem}  \
 Let $A$ and $B$ be invertible positive operators with $A-B\geq m>0$.
Then
$$A^{1+r}-(A^{\frac r2}B^pA^{\frac r2})^{\frac {1+r}{p+r}}
\geq m {m_A}^{r}$$
holds for $p \ge 1$ and $r \ge 0$.
\end{theorem}

\begin{proof} \
The conclusion for $r \in [0,1]$ is ensured by Theorem 4.1 because $k(b,m,p,\frac {p+r}{1+r}, r) = m$.
In particular, taking $r=1$, we have
$$A^2-(A^{\frac 12}B^pA^{\frac 12})^{\frac 2{p+1}} \ge m m_A:=m_1.
$$
Since $A_1=A^2$ and $B_1=(A^{\frac 12}B^pA^{\frac 12})^{\frac 2{p+1}}$ satisfy $A_1 - B_1 \ge m_1>0$, we have, for an arbitrary $s \in [0.1]$,
$$
A_1^{1+s} - (A_1^{\frac {1+s}2}B_1^pA_1^{\frac {1+s}2})^{\frac {1+s}{p_1+s}} \ge m_1,\ \text{where} \ p_1=\frac {p+1}2.
$$
Namely it is proved that if $A-B\ge m>0$, then
\begin{equation}  
A^{2(1+s)} - (A^{\frac {1+2s}2}B^pA^{\frac {1+2s}2})^{\frac {2(1+s)}{p+1+2s}} \ge m_1 {m_{A_1}}^s= m{m_A}^{1+2s},
\end{equation}
that is, putting $r=1+2s$, the conclusion holds for $r \in [1,3]$.

For the next step, if we put $s=1$ in (2) and $A_2=A^4$, $B_2= (A^{\frac 32}B^pA^{\frac 32})^{p_2}$, where $p_2 = \frac {p+3}4$, then $A_2 - B_2 \ge m {m_A}^3:=m_2$ holds and for an arbitrary $s \in [0,1]$,
$$
A_2^{1+s} - (A_2^{\frac {1+s}2}B_2^pA_2^{\frac {1+s}2})^{\frac {1+s}{p_2+s}} \ge m_2 {m_{A_2}}^s.
$$
Hence, putting $r=3+4s$, we obtain that the conclusion holds for $r \in [3,7]$.
By repeating this method, we can prove the theorem.
\end{proof}

\begin{corollary}  \
Let $A$ and $B$ be invertible positive operators with $A-B\geq m>0$.
Then for each $r \ge0$
$$A^{\frac {p+r}q}-(A^{\frac r2}B^pA^{\frac r2})^{\frac 1q}
\geq \|A\|^{\frac {p+r}q} - (\|A\|^{1+r}- m{m_A}^r)^{\frac {p+r}{q(1+r)}}
$$
holds for $p \ge 0$ and $q \ge 1$ with
$ (1+r)q \ge p+r$.
\end{corollary}

\begin{proof}
Since $\alpha = \frac {p+r}{q(1+r)} \le 1$, we apply Theorem A to $A_1$ and $A_2$ such that
$$A_1=A^{1+r}> B_1 = (A^{\frac r2}B^pA^{\frac r2})^{\frac {1+r}{p+r}}$$
by Theorem 4.2.
Then we have
$$  A_1^\alpha - B_1^\alpha \ge \|A_1\|^\alpha - (\|A_1\|-m{m_A}^r)^\alpha.
$$
So we get the desired lower bound.
\end{proof}

\end{document}